\documentclass[12pt]{amsart}
\usepackage{latexsym, amsmath, amsthm}

{\theoremstyle{plain}
    \newtheorem{Thm}{\bf Theorem}[section]
    \newtheorem{Prop}[Thm]{\bf Proposition}

    \newtheorem{Cor}[Thm]{\bf Corollary}
    
    \newtheorem{Conj}[Thm]{\bf Conjecture}
}
{\theoremstyle{remark}
    \newtheorem{Rem}[Thm]{\bf Remark}
    \newtheorem{Exa}[Thm]{\bf Example}
}
{\theoremstyle{definition}
    \newtheorem{Def}[Thm]{\bf Definition}
}
\numberwithin{equation}{section}

\newcommand{\fm}{{\mathfrak m}}
\newcommand{\m}{{\mathfrak m}}

\newcommand{\ringR}{\text{$(R,\fm,k)$ }}

\newcommand{\hr}{{H_{\fm}^{d}(R)}}

\newcommand{\excise}[1]{}

\title{A finiteness condition on local cohomology in positive characteristic}
\author[F.~Enescu]{Florian Enescu}

\address{Department of Mathematics and Statistics, Georgia State University, Atlanta, 30303}
\email{fenescu@gsu.edu}
\thanks{The author was partially supported by the NSA Young Investigator Grant H98230-10-1-0166.}
\begin{document}
\begin{abstract}
In this paper we present a condition on a local Cohen-Macaulay F-injective ring of positive characteristic $p > 2$ which implies that its top local cohomology module with support in the maximal ideal has finitely many Frobenius compatible submodules.
\end{abstract}

\maketitle
\markboth{A finiteness condition on local cohomology in positive characteristic}{A finiteness condition on local cohomology in positive characteristic}

\section{Introduction}

In this note  $\ringR$ denotes a local ring of positive characteristic $p >0$ and dimension $d$. Let $F: R \to R$ defined by $F(r) =r^p$ for all $r \in R$ be the Frobenius homomorphism on $R$.  Let $M$ be an $R$-module and $F_M=F: M \to M$ a Frobenius action on $M$, that is, $F$ is additive and $F(rm) = r^pF(m)$, for all $r \in R, m \in M$. Following terminology that has been gaining ground recently, we will say that a $R$-submodule $N$ of $M$ is F-compatible if $F(N) \subseteq N$. 

In the past two decades, applications of the Frobenius homomorphism led to important contributions in commutative algebra. Among the remarkable classes of rings related to Frobenius, F-injective and F-pure rings hold a significant place. A ring $R$ is called F-pure if $F$ is a pure homomorphism. The Frobenius homomorphism induces an action on the local cohomology modules with support in the maximal ideal of $R$. A ring $R$ is called F-injective if $F$ acts injectively on all modules $H^i_{\fm}(R)$, $i=1, \ldots, d$. This holds if $R$ is F-pure. A Gorenstein ring $R$ is F-pure if and only if it is F-injective. For a more detailed account of these claims, we refer the reader to~\cite{EH}. The reader might also find~\cite{FW} useful for these topics.

Rings with  the property that the local cohomology modules $H^i_\fm(R)$ have finitely many F-compatible submodules have been studied in~\cite{EH}. In fact, a ring $R$ is called FH-finite if for all $i=1, \ldots, d$, $H^i_\fm (R)$ contains finitely many F-compatible submodules. A result obtained by Enescu and Hochster in~\cite{EH} states that an F-injective Gorenstein ring is FH-finite.

In this paper we prove the following result (see Corollary~\ref{maincor}). 

\begin{Thm}
\label{main}
Let $\ringR$ be a local F-injective Cohen-Macaulay ring of prime characteristic $p >2$  that admits a canonical module. Let $I$ an ideal of $R$
isomorphic to the canonical module of $R$. Assume that $R/I$ is F-injective. 

Then $R$ is FH-finite.
\end{Thm}

It should be noted that the hypothesis of the Theorem~\ref{main} implies that $R$ is F-pure, see Corollary 2.5 in~\cite{E}. It is perhaps natural to conjecture that a more general statement holds.

\begin{Conj}
\label{conj}
Let $\ringR$ be a Cohen-Macaulay F-pure ring of dimension $d$. Then $\hr$ contains finitely many F-compatible submodules, that is $R$ is FH-finite.
\end{Conj}

This question was already raised in Discussion 4.5 in~\cite{EH} although not as a conjecture. It was shown there that the statement can be reduced to the case when $R$ is Cohen-Macaulay, complete, F-finite and F-split. 

It should be observed that FH-finite local rings are closely connected to the notion of antinilpotency for local cohomology modules with support in the maximal ideal. We will need this notion later in the paper so we will mention some basic facts about it.  The module $M$ is called antinilpotent if $F$ is injective on $M/N$  for all F-compatible $R$-submodules $N$ of $M$.  Let $R_0=R, R_n = R[[x_1, \ldots, x_n]]$ for $n \in \mathbb{N}_{\geq 1}$. The following theorem was proved in~\cite{EH}. 

\begin{Thm}
\label{anti}
Let $\ringR$ be a local ring of dimension $d$. Then $R_n$ is FH-finite for all $n \in \mathbb{N}$ if and only if $H^i_\fm(R)$ is antinilpotent for all $i=0, \ldots, d$.
\end{Thm}

For more details on antinilpotent modules we refer the reader to Section~4 in~\cite{EH}.

The author was informed that, prompted by the posting of a preprint version of this paper on the arXiv, L.~Ma has produced a different proof of Theorem~\ref{main} that also covers the case $p=2$.  The author thanks the referee and K.~Schwede for expository suggestions.

\section{The F-injectivity of the pseudocanonical cover and FH-finite rings}

Let $\ringR$ be a local, Cohen-Macaulay with canonical module $\omega_R$.
Embed $\omega _R$ into $R$ (which can be done when $R$ is generically Gorenstein) and let $I$ be 
this ideal of $R$ isomorphic to $\omega_R$. For standard facts on the canonical module for Cohen-Macaulay rings that will be used in this
section we refer the reader to Chapter 3 in~\cite{BH}.

\begin{Def}For $f \in R$, let $S= S(f)= R + It$ the subring of $R[T]/(T^2-f)$, where $t$ is the image of $T$ in the quotient.
We call $S$ a {\it pseudocanonical cover} of $R$ {\it via} $f$. 
\end{Def}
It should be noted that this definition was presented in~\cite{E} under the assumption
that $f \in \fm$. Here, we will allow $f \in R$.

The main properties of  $S=S(f)$ are summarized in the following result obtained in~\cite{E}.

\begin{Thm}
\label{properties} Under the notations and assumptions just introduced, 
$S$ is local, with maximal ideal equal to $\fm + It$;
$S$ is domain if and only if $R$ is domain and $f$ is not a square  in the total ring of fractions of $R$; any $x_1,...,x_k$ regular sequence on $R$ forms a regular sequence on $S$:
$S$ is Gorenstein with the socle generator for $S/\underline{x}S$ given by $ut$, where $u$ is the socle generator of $I/\underline{x}I$ and 
$\underline{x}=x_1,...,x_d$ form a s.o.p. on $R$.
\end{Thm}

\begin{proof}

Although this result was proved~\cite{E} only in the case $f \in \fm$, the statements can be proved with essentially the same proof.  Since we will be using
our construction for the case when $f$ is a unit, we will include a proof, at times more elementary,

We would like to give proofs to each of these statements:

{\it $S$ is local, with maximal ideal equal to $\m + It$:}

We will show that $S \setminus (\m+It)$ consists of units in $S$. Let $a+bt \in S$ such that $a$ is not in $\m$. So, $a$ is unit $R$.
We are looking for $c, d$ such that $(a+bt)(c+dt) =1$. That is, $ac+bdf =1$ and $ad+bc=0$. Using that $a$ is invertible in $A$, we can solve for
$c$ in the first equation and plug into the second obtaining that $d(b^2f-a^2) = b$. Observe that $f \in \m$ implies that
$b^2f \in \m$ and so $b^2f-a^2$ is unit in $R$. In conclusion, $d= \frac{b}{b^2f-a^2}$ and, a simple computation shows that $c =\frac{a}{b^2f-a^2}$.

{\it $S$ is domain if and only if $R$ is a domain and $f$ is not a square in the total ring of fractions of $R$:}

As in the paragraph above, $(a+bt)(c+dt) =0$ is equivalent to $ac +bdf =0$ and $ad+bc=0$. Eliminating $c$, we get $d(b^2f-a^2)=0$
and then similarly $c(b^2f-a^2) =0$. One direction is now clear. For the other direction, assume that $S$ is domain. Then $R \subset S$ is also domain.
If $f = b^2/a^2$ for some $a, b$ in $R$, we can rewrite $f = (bx)^2/(ax)^2$, for $0 \neq x \in I$ and then $(ax + bxt)(ax-bxt)=0$ which shows that 
$S$ is not domain.

{\it A regular sequence $x_1,...,x_k$ on $R$ is a regular sequence on $S$:}

Let $1 \leq i \leq k$ be an integer. Assume that $x_i \cdot (a+bt) = x_1(a_1 + b_1t)+ \cdots + x_{i-1}(a_{i-1}+b_{i-1}t)$ where $a, a_j \in R$
and $b, b_j \in I$, for all $1 \leq j \leq i-1$. So, $x_i a = \sum x_j a_j$ and hence $x_i \in (x_1, \cdots, x_{i-1})R \subset (x_1, \cdots, x_{i-1})S$.

{\it S is Gorenstein:}
Using the notations of the previous paragraph, denote $J = (x_1,...,x_d)$.

Let $u \in I$ be such that its image in $I/JI$ generates the socle of $I/JI$. In particular, $\fm u \subset JI$.

Let $\overline{ut}$ be the class of $ut$ in $S/JS$. Clearly $(\fm+It) ut = Ifu + \fm u t \subset  JI + JI t \subset JR + JI t=JS$, so $\overline{ut} \in Soc(S/JS)$.

We prove now that $\overline{ut}$ generates the socle of $S/JS$. Let $a \in R$, $b \in I$. If $(\fm +It) (a+bt) \in JS$, then
$$\fm a +bfI \in JR,$$ and
$$Ia + \fm b \in JI.$$
In particular $\fm b \in JI$, and, since $u$ is generator for the socle of $I/JI$, we get that $ b \in uR + JI$. Also, $Ia \subset JI$ says that
$a$ kills the module $I/JI$ which is a faithful module over $R/J$, so $a \in J$. 

Hence, $a+bt \in JR + utR + JIt$, and then $\overline{a+bt}$ is a multiple of $\overline{ut}$ in $S/JS$.
\end{proof}

We would like to investigate the conditions on $R$ and $f$ that make $S(f)$ F-injective. Since $S(f)$ is Gorenstein and hence Cohen-Macaulay, the $S(f)$ is F-injective if and only if any (equivalently, for some) ideal $J$ generated by a system of parameters in $S(f)$ is Frobenius closed. So,  we need
to find the conditions which imply that an ideal of $S(f)$ generated by a system of parameters is Frobenius closed

\begin{Prop}
\label{one}
Let $(R, m, k)$ be a local Cohen-Macaulay ring of prime characteristic $p >0$ that admits a canonical ideal $I$. Let $\underline{x}= x_1, x_2,...,x_d$ be 
system of parameters in $R$ and $u$ the socle generator of $I/\underline{x}I$.
 
\item (i) If $p >2$, then there exists $q$ large enough such that $u^qf^{(q-1)/2} \in (x_1^q, x_2^q,...,x_d^q)I$ if and only if $S(f)$ is not $F$-injective.

\item (ii) If $p=2$, $S(f)$ is never $F$-injective.

\end{Prop}

\begin{proof}
As earlier denote $S:=S(f)$. For an s.o.p. $x_1, \cdots, x_d$ in $R$, let
$J=(x_1,...,x_d)$. Then $JS$ is a parameter ideal in $S$ and $S$ is Gorenstein.

If $JS$ is not Frobenius closed, then a multiple of $ut$, the socle generator modulo $JS$, must belong to $(JS)^F \setminus JS$.
This implies that $ut$ itself belongs to $(JS)^F \setminus (JS)$. In conclusion, $ut \in (JS)^F \setminus (JS)$ if and only if
$S$ is not $F$-injective.

First, let us treat the case $p \neq 2$. 

Note that for $a+bt \in S$, we have that $(a+bt)^q = a^q + b^q t^q = a^q + b^qf^{(q-1)/2}t$. Let $q$ such that $(ut)^q \in (JS)^{[q]}$. 
By taking Frobenius powers, we can assume that $q$ is large enough. So, $u^qf^{(q-1)/2}t= \sum_{i=1}^d x_i^q (a_i+b_it)$ with $a_i \in R$ and $b_i \in I$.
Hence $u^qf^{(q-1)/2}= \sum_{i=1}^d x_i^q b_i$ (and $\sum_{i=1}^d x_i^q a_i =0$, which can be easily arranged independent of $u$). This proves 
the first part.

Let us look at the remaining case $p=2$. Since $q=2^e$, $(a+bt)^q = a^q + b^q t^q= a^{2^e} + b^{2^e} f^{2^{e-1}}$ which represents an element of $R$.
So, $(ut)^q \in (JS)^{[q]}$ is equivalent to $u^{2^e}f^{2^{e-1}} = \sum_{i=1}^d x_i^{2^e}(a_i+b_it)$. It follows that 
$u^{2^e}f^{2^{e-1}} = \sum_{i=1}^d x_i^{2^e}a_i$. 

For a more precise illustration, take $e=1$. Then $(ut)^2= u^2f$. But $\fm u \subset JI$ so $u$ kills $I/JI$ and hence $u \in J$. In conclusion, $u^2 \in J^{[2]} \subset J$ and hence
$(ut)^2 \in J^{[2]}S$.

\end{proof}

\begin{Prop}
\label{two}
Let $\ringR$ be a generically Gorenstein Cohen-Macaulay ring and $I \subset R$ an ideal isomorphic to the canonical module of $R$, $\omega_R$.
Let $x=x_1 \in I$ be a nonzerodivisor  (NZD) on $R$ and $x_2,...,x_d$ elements in $R$ that form a regular sequence on $R/I$ and $R/xR$. Denote $J = (x_1,...,x_d)$. Then 
$R/I$ is a maximal Cohen-Macaulay module of type $1$ 
over $R/xR$ and $Soc(\frac{R}{I +(x_2, \cdots, x_d)R})$ is canonically isomorphic to $Soc(\frac{I}{JI})$.
\end{Prop}

\begin{proof}
It is known that $R/I$ is Gorenstein, so let $z \in R$ such that $\overline{z}$ is the socle generator of $\frac{R}{I +(x_2, \cdots, x_d)}$.

Note that there is a natural $R$-linear map, $R/I \stackrel{\cdot x}{\to} I/xI$ which is, in fact, injective because
$x$ is a NZD on $R$.

Let us prove now that $Soc(\frac{R}{I +(x_2, \cdots, x_d)R})$ is mapped injectively into $Soc(\frac{I}{JI})$:

Let $z \in R \setminus (I +(x_2, \cdots, x_d)R)$ such that $m z \subset I +(x_2, \cdots, x_d)R$. 
Then $\fm \cdot xz \subset xI + x(x_2, \cdots, x_d)R \subset JI$.

To prove that $xz$ is not in $JI$ it is enough to note that if $xz = xa_1 +x_2a_2 + \cdots + x_d a_d$ with $a_i \in I$, then
$z -a_1 \in (x_2,...,x_d) R$, since $x, x_2, ...., x_d$ are a regular sequence on $R$. In conclusion, $z \in I +(x_2, \cdots, x_d)R$
which is a contradiction.

Now, $Soc(\frac{R}{I +(x_2, \cdots, x_d)R})$ maps injectively into $Soc(\frac{I}{JI})$. Both are $1$-dimensional over $R/\fm$ and hence
must be isomorphic. This proves the assertion.
\end{proof}

\begin{Thm}
Let $\ringR$ be a local F-injective Cohen-Macaulay ring of prime characteristic $p >2$ and dimension $d \geq 1$ that admits a canonical module. Let $I$ be an ideal of $R$
isomorphic to the canonical module of $R$.

Assume that $R/I$ is F-injective. Then $S=S(1)$ is F-injective.

\end{Thm}

\begin{proof}
Choose $x =x_1$ be a nonzerodivisor on $R$ contained in $I$.

We can choose $x_2,...,x_d$ elements in $R$ that form a regular sequence on $R/I$ and $R/xR$. Denote $J = (x_1,...,x_d) = (\underline{x}).$
Let $u \in I$ such that its maps to the socle generator of $I/(\underline{x})I$.

Assume that $S(1)$ is not F-injective. By our Proposition~\ref{one}, for each $f$ there exists $q=q(f)$ large enough such that $u^q \in (x^q, x_2^q,...,x_d^q)I$.

Now $u = xz + v= x_1 z+v $ for some $z \in R$ mapping to the socle of $\frac{R}{I +(x_2, \cdots, x_d)R}$ and 
$v \in (\underline{x})I$, by Proposition~\ref{two}. Also, $z$ is not in $I +(x_2, \cdots, x_d)R$ since its image generates the socle of $R$ modulo $I +(x_2, \cdots, x_d)R$.
So, $x^q z^q \in (x^q, x_2^q,...,x_d^q)I$. 

Since $\underline{x}$ form a regular sequence on $R$, we get that
$z^q \in I +(x_2^q, \ldots, x_d^q)$. But $x_2, \ldots, x_d$ generate parameter ideal in $R/I$.

Therefore $z \in I+(x_2, \ldots, x_d)$ since $R/I$ is F-injective. This is a contradiction.

\end{proof}

\begin{Cor}
\label{maincor}
Let $\ringR$ be a local F-injective Cohen-Macaulay ring of prime characteristic $p >2$  that admits a canonical module. Let $I$ an ideal of $R$
isomorphic to the canonical module of $R$. Assume that $R/I$ is F-injective. 

Then $R$ is FH-finite and $H^d_m(R)$ is antinilpotent.

\end{Cor}

\begin{proof}
We know that $S(1)$ is FH-finite. This descends to $R$ since $R \to S(1)$ is $R$-split by Lemma 2.7 in~\cite{EH}.

The hypothesis are preserved if we pass to $R[[x]]$ because a canonical ideal for it is equal to $I[[x]]$ (Corollary 3.3.21 in~\cite{BH} applies here directly). So, we get that $H^d_m(R)$ is in fact antinilpotent, since $R[[x]]$ is FH-finite by Theorem~\ref{anti}.
\end{proof}

\begin{Rem}
It can be noted that the hypothesis of the Conjecture~\ref{conj} is preserved by passing to a formal power series so an equivalent statement is that for a Cohen-Macaulay F-pure ring $R$
 the module $\hr$ is antinilpotent. As we mentioned earlier, when $R$ is Gorenstein, the conjecture is known and, in that case, $\hr$  is naturally isomorphic to $E_R(k)$. Therefore, the reader might wonder whether
it is natural to formulate a statement similar to the Conjecture~\ref{conj} for $E_R(k)$ instead of $\hr$.
 
 In response to this question, one can note that, for a Cohen-Macaulay F-pure ring $R$,
 the injective hull of its residue field, $E_R(k)$ admits a Frobenius action such that $E_R(k)$ is antinilpotent, by a result of Sharp (Lemma 3.1 in~\cite{Sh} and Proposition 1.8 in~\cite{Sh2}).
\end{Rem}

\begin{Exa}
We would like to illustrate our corollary with a nontrivial example due to Katzman, see~\cite{K} where the example is considered over $\mathbb{F}_2$. Let $k=\mathbb{F}_3$ and let $S= k[[x_1,\ldots, x_5]]$. Let
$\mathcal{I}$ be the ideal generated by the $2 \times 2$ minors of 

\[ 
(\begin{array}{cccc}
x_1 & x_2 & x_2 & x_5 \\
x_4 & x_4& x_3 & x_1 
\end{array} ).
\]

Consider $R = S/\mathcal{I}$. The ring $R$ is Cohen-Macaulay reduced and two dimensional. The ideal $I=(x_1, x_4, x_5)$ is a canonical ideal. The quotient $R/I = \mathbb{F}_3[x_2, x_3]/(x_2x_3)$ is F-pure. We have checked
these claims using Singular.

According to our main result, the ring $R$ is FH-finite.
\end{Exa}

Finally we would like to observe the following consequence of Theorem 4.21 in~\cite{EH}.

\begin{Rem}
Let $\ringR$ be Cohen-Macaulay F-pure such that the test ideal of $R$ is $\fm$-primary. Then $R$ is FH-finite. In particular Conjecture~\ref{conj} holds for one dimensional rings.
\end{Rem}

The proof of this remark is an immediate consequence on Theorem 4.21 in~\cite{EH} as follows. First complete $R$ then enlarge the residue field of $R$ such that it is perfect. The test ideal $\tau$ remains $\fm$-primary and satisfies
$\tau I^* \subseteq I$ for all ideals $I$ generated by parameters. Moreover, for one-dimensional complete rings the test ideal exists and is automatically $\fm$-primary.

\begin{Rem}
The conjecture also holds for Stanley-Reisner rings as shown in Theorem 5.1 in~\cite{EH} which in fact contains a more general statement.
\end{Rem}

\end{document}